\documentclass[a4paper,twoside,11pt]{article}

\usepackage{color,soul}
\usepackage{xcolor}

\usepackage[most]{tcolorbox}
\newtcolorbox{highlighted}{colback=yellow,coltext=red,breakable}
\usepackage{lipsum}
\usepackage{appendix}

\newcommand{\paragraphNew}[1]{\paragraph{\underline{#1}}}
\usepackage{multirow}
\usepackage{authblk}

\usepackage[english]{babel}
\usepackage{lmodern}
\usepackage[T1]{fontenc}

\usepackage[a4paper,top=3cm,bottom=3cm,left=3cm,right=3cm,
bindingoffset=5mm]{geometry}

\usepackage{amsmath,amssymb,amsthm}
\usepackage{graphicx}

\usepackage{epstopdf}

\usepackage{subfigure}
\usepackage{epstopdf}
\graphicspath{{./Figures/}}
\graphicspath{{./Figures/}}

\newtheorem{thm}{Theorem}[section]
\newtheorem{prop}[thm]{Proposition}
\newtheorem{lem}[thm]{Lemma}

\newtheorem{remark}{Remark}

\renewcommand{\O}{\Omega}

\newcommand\E{{E}}  

\newcommand\Th{{\mathcal T}_h}

\newcommand\Fh{{\mathcal F}_h}

\newcommand\R{\mathbb{R}}
\newcommand\C{\mathbb{C}}
\newcommand\D{\mathbb{D}}

\renewcommand{\P}{{\mathcal P}}  

\def\decapita#1{}
\def\grecomultibold#1#2{\grecobolddef#1\def\secondobold{#2}%
	\ifx#2\finemultibold\let\next\relax\let\secondobold\relax
	\else\let\next\grecomultibold
	\fi\expandafter\next\secondobold}

\def\grecobolddef#1{%
	\edef\dadef{bf\expandafter\decapita\string#1}%
	\expandafter\def\csname\dadef\endcsname{{\neretto #1}}}

\def\neretto#1{\setbox0=\hbox{\mathsurround=0pt$#1$}%
	\kern.02em\copy0 \kern-\wd0
	\kern-.02em\copy0 \kern-\wd0
	\raise.03em\box0 \kern.02em}
\grecomultibold\alpha\phi\beta\omega\gamma\delta\xi\sigma\tau\eta\theta\varepsilon\rho\psi\pi\varphi\varepsilon\finemultibold
\def\bdiv{\mathop{\bf div}\nolimits}

\def\teps{{\bfvarepsilon}}

\def\T{{\mathcal T}}

\def\Ih{{\mathcal I}_h}
\def\IE{{\mathcal I}_E}

\def\wbox#1;#2;{\vbox{\hrule\hbox{\vrule height#1mm\kern#2mm\vrule
			height#1mm}\hrule}}

\let\phi\varphi

\newcommand{\bbf}   {\mathbf{f}}

\newcommand{\bbn}   {\mathbf{n}}

\newcommand{\bbp}   {\mathbf{p}}

\newcommand{\bbr}   {\mathbf{r}}

\newcommand{\bbt}   {\mathbf{t}}
\newcommand{\bbu}   {\mathbf{u}}
\newcommand{\bbv}   {\mathbf{v}}
\newcommand{\bbw}   {\mathbf{w}}
\newcommand{\bbx}   {\mathbf{x}}

\def\C{\mathbb C}

\def\hpoint#1.#2.#3{{\underline{#1}}_{#2}\cdot
	{\underline{\mathop{\smash{#3}\vphantom{{#1}_{#2}}}}}}

\def\npointp#1.#2{{\underline{#1}}\cdot
	{\underline{\mathop{\smash{#2}\vphantom{{#1}}}}}}

\def\beq{\begin{equation}}
\def\enq{\end{equation}}

\def\bfzero{{\bf 0}}
\def\xt{{\tilde x}}
\def\yt{{\tilde y}}
\def\bxt{{\tilde \bbx}}
\newcommand{\dO}    {\text{d}\Omega}
\newcommand{\dEl}    {\text{d} E}
\newcommand{\df}    {\text{d} f}

\usepackage{todonotes}

\definecolor{ambre(sae/ece)}{rgb}{1.0, 0.49, 0.0}
\definecolor{capri}{rgb}{0.0, 0.75, 1.0}
\newcommand{\coorF}[1]{{\bf {\color{capri} #1}}}

\author[1]{F. Dassi\thanks{{\tt franco.dassi@unimib.it}}}
\author[2,3]{C. Lovadina\thanks{{\tt carlo.lovadina@unimi.it}}}
\author[1]{M. Visinoni\thanks{{\tt michele.visinoni@unimib.it}}}
\affil[1]{{\footnotesize Dipartimento di Matematica e Applicazioni, Universit\`a degli studi di Milano Bicocca, Via Roberto Cozzi 55 - 20125 Milano, Italy}}
\affil[2]{{\footnotesize Dipartimento di Matematica, Universit\`a di Milano, Via Saldini 50, 20133 Milano, Italy}}
\affil[3]{{\footnotesize IMATI del CNR, Via Ferrata 1, 27100 Pavia, Italy}}
\date{}

\title{A three-dimensional Hellinger-Reissner \\ Virtual Element Method \\ for linear elasticity problems}

\begin{document}
	
	\maketitle

	\begin{abstract}
		We present a Virtual Element Method for the 3D linear elasticity problems, based on Hellinger-Reissner variational principle.
		In the framework of the small strain theory, we propose a low-order scheme with a-priori symmetric stresses and continuous tractions across element interfaces. 
		A convergence and stability analysis is developed and we confirm the theoretical predictions via some numerical tests.
		
	\end{abstract}

	{
		\section{Introduction}\label{s:intro}
		The Virtual Element Method (VEM), introduced in~\cite{volley,hitchhikers}, is a recent technology for the approximation of partial differential equation problems. This method is a generalization of the Finite Element Method (FEM) which allows to deal with arbitrary polygonal/polyhedral meshes, also including non convex and distored elements.
		To garantee this flexibility, the virtual element method abandons the idea of the local polynomial approximation, typical of FEM, to use approximating functions which are solution of suitable local PDE.
		In general, these non-polynomial functions are not explicitly known. 
		Therefore, the main idea of this method is to exploit the available information (the degrees of freedom) to compute the stiffness matrix and the right-hand side of the discretized problem.
		%
		
		%
		During these years, VEM have been employed with success both in mathematical and engineering communities.
		Here we mention, as a rapresentative non-exaustive sample, a brief list of papers~\cite{BBDMR_familyMagnetos,HighOrder3D,BLRXX,BLV,VEM-curved,bertoluzza2017,BRENNERVEM1,BRENNERVEM2,BCP,DASSI2019,MascottoDassi,Mascotto2018}.
		In the framework of structural mechanics problems, we cite the recent works~\cite{Andersen2017,ABLS_part_I,ABLS_part_II,ARTIOLI2017155,ARTIOLI2018978,ARTIOLI_RICOVERYVEM,wriggers2017,wriggers,ZHANG20181} and~\cite{BeiraodaVeiga-Brezzi-Marini:2013,lovadina,Brezzi-Marini:2012,Paulino-VEM}, for instance. However, we remark that VEM is not the only technology which can make use the polytopal meshes. Considering elasticity problems, we mention~\cite{bottiDipietroSochala,cockburn2017,DiPietro-Ern-1,cockburn2015,prada2017} as representative examples.
		In the present paper we extend the study presented in~\cite{ARTIOLI2017155} to the three dimensional case. More precisely, we design and analyze a low-order virtual element method for linear elasticity problems. Within the framework of small displacements and small deformations, we consider the Hellinger-Reissner variational principle as the basis of our discretization procedure. This mixed formulation describes the problem by means of both the displacement and the stress fields. In the Finite Element practice it is a difficult task designing a stable and accurate scheme that preserves both the symmetry of the stress tensor and the continuity of the tractions at the inter-elements, see for instance \cite{BoffiBrezziFortin} and \cite{ArnoldWinther}. The fundamental reason behind this difficulty lies in the local polynomial approximation, which forces the introduction of nodal degrees of freedom for the stress unknown. The resulting finite element schemes are typically quite cumbersome, especially in the three dimensional case, see \cite{ArnoldAwanouWinther}. Furthermore, the presence of nodal degrees of freedom introduces an additional complication if one aims at using the hybridization procedure to solve the discrete linear system, see \cite{ArnoldBrezzi_1985}.
		We exploit the flexibility of virtual element methods to avoid these drawbacks and to develop an optimal scheme which is reasonably cheap with respect to the delivered accuracy. Since the approximated stresses does not have nodal degrees of freedom (on the contrary, the degress of freedom are entirely local to each polyhedron face), the hybridization procedure could be easily applied to our VEM scheme. This aspects show that, even for tetrahedral or hexahedral meshes, the proposed VEM method is a valid alternatives to FEM schemes.  
		The paper is organized as follows. In Section~\ref{sec:1} we briefly introduce the classical Hellinger-Reissner formulation of the 3D elasticity problem.
		Section~\ref{s:HR-VEM} describes the Virtual Element approximation we propose, while Section~\ref{s:theoretical} is about the convergence analysis of the method.
		The numerical experiments, which confirm the theoretical predictions, are detailed in Section~\ref{s:numer}. Finally, we draw some conclusions.
		\paragraph{Space notation.}
		Throughout the paper, we will make use of standard notations regarding Sobolev spaces, norms and seminorms (cf. \cite{Lions-Magenes} for example). In addition, given two quantities $a$ and $b$, we write $a\lesssim b$ when there exists a constant $C$, indipendent of the meshsize, such that $a\leq C\, b$. 
		Finally, given any subset $A\subset\R^n$ and an integer $k\geq 0$, $\P_k(A)$ denotes the space of polynomials up to degree $k$, defined on $A$; whereas, given a functional space $X$, we denote by $\left[X\right]^{3\times 3}_s$ the $3 \times 3$ symmetric tensors whose components belong to the space $X$.
		\paragraph{Mesh notation.}
		Given a polyhedron $E$ with $n_f^E$ faces we denote its volume, diameter and barycenter by $|E|$, $h_E$ and $\bbx_E$, respectively. In a similar way we refer to the area, diameter and barycenter of a face $f$, while $|e|$ denotes the length of the edge $e$. 		
		Given a polygonal face $f$, we use $\bbx$ and $\bxt$ to indicate the global and local coordinates of a generic point of $f$, respectively.
		\section{The elasticity problem in mixed form}\label{sec:1}
		In this section we briefly present the 3D elasticity problem based on Hellinger-Reissner principle~\cite{BoffiBrezziFortin, Braess:book}:
		\begin{equation*}\label{strong}
		\left\lbrace{
			\begin{aligned}
			&\mbox{Find } (\bfsigma,\bbu)~\mbox{such that}\\
			&-\bdiv \bfsigma= \bbf\  &\mbox{in $\Omega$}\\
			& \bfsigma = \C \teps(\bbu)\ &\mbox{in $\Omega$}\\
			&\bbu=\bfzero\ &\mbox{in $\partial\Omega$}
			\end{aligned}
		} \right.
		\end{equation*}
		where $\O\subseteq\R^3$ is a polyhedral domain, $\bfsigma$ and $\bbu$ represent the stress and the displacement fields, respectively.
		We define the bilinear form $a(\bfsigma,\bftau):=(\D \bfsigma, \bftau)$, where $(\cdot,\cdot)$ is the scalar product in $L^2$, and $\D:=\C^{-1}$. Then, the mixed variational formulation reads:
		\begin{equation}\label{cont-pbl}
		\left\lbrace{
			\begin{aligned}
			&\mbox{Find } (\bfsigma,\bbu)\in \Sigma\times U~\mbox{such that}\\
			&a(\bfsigma,\bftau) + (\bdiv \bftau, \bbu)=0  &\forall \bftau\in \Sigma\\
			& (\bdiv \bfsigma, \bbv) = -(\bbf,\bbv) &\forall \bbv\in U
			\end{aligned}
		} \right.
		\end{equation}
		where $\bbf \in \left[L^2(\O)\right]^3$ is the loading term, while the spaces $U$ and $\Sigma$ are $$U=\left[L^2(\Omega)\right]^3$$
		and
		$$
		\Sigma=\left\{ \bftau\in H(\bdiv;\Omega)\ :\ \bftau \mbox{ is symmetric} \right\}.
		$$
		As usual, $H(\bdiv;\Omega)$ is the space of tensor in $\left[L^2(\O)\right]^{3\times 3}$ whose divergence is the vector-valued operator in $\left[L^2(\O)\right]^3$. 
		The elasticity fourth-order symmetric tensor $\D$ is assumed to be uniformly bounded and positive-definite. 
		
		It is well known that Problem~\eqref{cont-pbl} is well posed, cf. \cite{BoffiBrezziFortin} for example. In particular, considering the natural norms
		\begin{equation*}
		||\bfsigma||_{\Sigma}^2:=\int_{\O}|\bfsigma|^2~\dO+\int_{\O}|\bdiv(\bfsigma)|^2~\dO,\qquad ||\bbu||^2_{U}:=\int_{\O}|\bbu|^2~\dO,
		\end{equation*}
		it holds:
		$$
		||\bfsigma||_{\Sigma}+||\bbu||_{U}\leq C||\bbf||_{0}
		$$
		where C is a constant depending on $\O$ and on the material tensor $\D$, which, however, does not degenerate in the incompressible limit. We also remark that a possible interesting variant of the variational formulation \eqref{cont-pbl} has been recently proposed and studied in \cite{schoeberl1} and \cite{schoeberl2}.
		
		When considering a polyhedral mesh $\T_h$ of the domain $\Omega$, the bilinear form $a(\cdot,\cdot)$ is split as
		\begin{equation*} a(\bfsigma,\bftau)=\sum_{{\E\in \Th}}a_E(\bfsigma,\bftau) \quad \textrm{ with } \quad
		a_E(\bfsigma,\bftau) : = \int_E \D \bfsigma : \bftau~\dEl
		\end{equation*}
		for all $\bfsigma,\bftau \in \Sigma$.
		Similarly, it holds
		\begin{equation*}
		(\bdiv \bftau, \bbv)=\sum_{{\E\in \Th}}(\bdiv \bftau, \bbv)_E \quad \textrm{ with } \quad
		(\bdiv \bftau,\bbv)_E : = \int_E \ \bdiv \bftau \cdot \bbv~\dEl ,
		\end{equation*}
		for all $(\bftau,\bbv) \in \Sigma\times U$.
		%
		%
		\section{The Virtual Element Method}\label{s:HR-VEM}
		In this section we define our Virtual Element discretization of Problem \eqref{cont-pbl}.
		Let $\{\mathcal{T}_h\}_h$ be a sequence of decompositions of $\Omega$ into general polyhedral elements $E$ with
		\[
		%
		h := \sup_{E \in \mathcal{T}_h} h_E.
		\]
		%
		We suppose that for all $h$, each element $E$ in $\mathcal{T}_h$ fulfils the following standard assumptions (cf. \cite{volley}):
		\begin{itemize}
			\item $\mathbf{(A1)}$ $E$ is star-shaped with respect to a ball of radius $ \ge\, \gamma \, h_E$,
			\item $\mathbf{(A2)}$ for every face $f\in\partial E$ we have  $h_f \ge \gamma h_E$ and $f$ is star-shaped with respect to a disk of radius $ \ge\, \gamma\, h_f$,
			\item $\mathbf{(A3)}$ for every edge $e\in\partial f$, we have $|e| \ge \gamma \, h_f\ge \gamma^2\, h_E$,
		\end{itemize}
		where $\gamma$ is a suitable positive constant. We also assume that the elasticity tensor $\D$ is piecewise constant with respect to the decomposition $\T_h$, i.e. $\D$ is constant on each polyhedron of the mesh~\cite{ARTIOLI2017155}.
		%
		\subsection{The local spaces}\label{ss:E-spaces}
		%
		To describe the local spaces employed in the VEM proposed scheme, we introduce two spaces: $RM(E)$ and $T_h(f)$.

		\paragraphNew{Space $RM(E)$.} It is the space of local infinitesimal rigid body motions: 
		\begin{equation}\label{eq:rigid}
		RM(E):=\left\{ \bbr(\bbx) = \bfalpha + \bfomega \wedge \big(\bbx -\bbx_E\big)\ \ \text{s.t.}\ \bfalpha,\ \bfomega \in\R^3 \right\},	
		\end{equation}
		whose dimension is $6$.
		%
		%
		\paragraphNew{Space $T_h(f)$.} For each face $f\in\partial E$, we introduce
		\begin{equation}\label{eq:face_approx}
		T_h(f):=\left\{ \bfpsi(\bxt)=\,\bbt_{f} + a\big[\bbn_{f}\wedge(\bbx(\bxt) -\bbx_f)\big]+ p_1(\bxt)\bbn_{f},\ \text{s.t.}\ a\in \R,\ p_1(\bxt)\in\P _1(f) \right\},
		\end{equation}
		where $\bbn_{f}$ the outward normal to the face $f$, and $\bbt_{f}$ is an arbitrary vector tangent to the face $f$. Above, $\bbx(\bxt)$ is the three dimensional position vector of a point on $f$, determined by the two local coordinates $\bxt$.
		The dimension of such a space is $6$: 
		\begin{itemize}
			\item The three dimensional tangent vector $\bbt$ is determined by a linear combination of two given linearly independent tangential vectors $\bbt_1$ and $\bbt_2$, i.e. 
			\begin{equation*}
			\bbt=b_1\bbt_1+b_2\bbt_2.
			\end{equation*}
			\item The rotational term $a\big[\bbn_{f}\wedge(\bbx(\bxt) -\bbx_f)\big]$ is determined by a single scalar value $a\in\R$.
			\item The polynomial $p_1(\bxt)\in\P_1(f)$ is a two variable polynomial with respect to the local face coordinate system so it is determined by three parameters, for instance:
			\begin{equation*}
			p_1(\bxt)=c_1+c_2(\xt-\xt_f)+c_3(\yt-\yt_f).
			\end{equation*}
		\end{itemize}
		Therefore, this space consists of vector functions whose tangential component is a 2D face rigid body motion (the first two terms of \eqref{eq:face_approx}), while the normal component is a linear two-variable polynomial (the last term of \eqref{eq:face_approx}). 
		\paragraphNew{Stress space.} 
		We are now ready to introduce our local approximation space for the stress field:
		\begin{equation}\label{eq:local_stress}
		\begin{aligned}
		\Sigma_h(E):=\Big\{ \bftau_h\in & H(\bdiv;E)\ :\ \exists\, \bbw^\ast\in \left[H^1(E)\right]^3 \mbox{ such that } \bftau_h=\C\teps(\bbw^\ast);\\ 
		&(\bftau_h\,\bbn)_{|f}\in T_h(f) \quad \forall f\in \partial E;\quad \bdiv\bftau_h\in RM(E) \Big\}.
		\end{aligned}
		\end{equation}
		We have the following Proposition.
		
		%
		\begin{prop}\label{pr:divFromdofs}
			Let $\bftau_h\in\Sigma_h(E)$, then $\bdiv\bftau_h$ is completely determined by $(\bftau_h\,\bbn)_{|f}$, with $f\in\partial E$ face of $E$.
			More precisely, setting (cf \eqref{eq:rigid})
			\begin{equation}\label{eq:div1}
			\bdiv\bftau_h = \bfalpha_E + \bfomega_E\wedge\big(\bbx -\bbx_E\big),
			\end{equation}
			it holds
			\begin{equation*}
			\bfalpha_E = \frac{1}{|E|} \left(\sum_{f\in\partial E}\int_f (\bftau_h\,\bbn)~\df\right),
			\end{equation*}
			and $\bfomega_E$ is the unique solution of the $3 \times 3$  linear system
			\begin{equation}\label{eq:div_dofs_omega}
			\int_E  \big(\bbx -\bbx_E\big)\wedge \left[\bfomega_E\wedge
			\big(\bbx -\bbx_E\big)\right]~\dEl
			=\sum_{f\in\partial E} \int_{f}  \big(\bbx -\bbx_E\big)\wedge (\bftau_h\bbn)~\df.
			\end{equation}
		\end{prop}
		\begin{proof}
			Since $\bftau_h\in\Sigma_h(E)$, we have $\bdiv\bftau_h\in RM(E)$ and $(\bftau_h\,\bbn)_{|f}=\bfpsi(\bxt)_{|f}\in T_h(f)$, for $f\in\partial E$. Then, denoting with $\bfvarphi:\partial E\to \R^3$ the function such that $\bfvarphi_{|f}:=\bfpsi(\bxt)_{|f}$, the integration by parts:
			\begin{equation}\label{eq:compat}
			\int_E \bdiv\bftau_h\cdot \bbr~\dEl =
			\int_{\partial E}\bfvarphi\cdot \bbr~\df \qquad \forall \bbr\in RM(E) ,
			\end{equation}
			allows to compute $\bdiv\bftau_h$ using the degrees of freedom of the space $T_h(f)$. 

			Testing \eqref{eq:compat} with constant functions $\bbr(\bbx)=\bfalpha$ and recalling \eqref{eq:div1}, we have  
			\begin{equation*}
			\int_E \bfalpha_E\cdot\bfalpha~\dEl= \int_{\partial E} \bfvarphi_{|f}\cdot\bfalpha~\df =  \Big(\sum_{f\in\partial E}\int_f (\bftau_h\,\bbn)~\df\Big)\cdot\bfalpha
			\qquad \forall \bfalpha \in\R^3.
			\end{equation*}
			Hence, we obtain
			\begin{equation*}
			\bfalpha_E = \frac{1}{|E|} \Big(\sum_{f\in\partial E}\int_f (\bftau_h\,\bbn)~\df\Big).
			\end{equation*}
			Now, we test \eqref{eq:compat} selecting $\bbr(\bbx) = \bfomega\wedge\big(\bbx -\bbx_E\big)$. We have
			\begin{equation}\label{eq:rotation}
			\int_E \left[\bfomega_E\wedge\big(\bbx -\bbx_E\big)\right]\cdot \left[\bfomega\wedge\big(\bbx -\bbx_E\big)\right]~\dEl=\sum_{f\in\partial E} \int_{f} \bfphi_{|f} \cdot \left[\bfomega\wedge\big(\bbx -\bbx_E\big)\right]~\df\quad \forall \bfomega\in\R^3,
			\end{equation}
			i.e.
			\begin{equation*}
			\bfomega\cdot \left(\int_E  \big(\bbx -\bbx_E\big)\wedge \left[\bfomega_E\wedge
			\big(\bbx -\bbx_E\big)\right]~\dEl \right)
			=\bfomega\cdot\left(\sum_{f\in\partial E} \int_{f}  \big(\bbx -\bbx_E\big)\wedge \bfphi_{|f}~\df\right)\quad \forall \bfomega\in\R^3.
			\end{equation*}
			We then infer that $\bfomega_E$ satisfies 
			\begin{equation*}
			\int_E  \big(\bbx -\bbx_E\big)\wedge \left[\bfomega_E\wedge
			\big(\bbx -\bbx_E\big)\right]~\dEl
			=\sum_{f\in\partial E} \int_{f}  \big(\bbx -\bbx_E\big)\wedge (\bftau_h\bbn)~\df,
			\end{equation*}
			i.e. $\bfomega_E$ solves system \eqref{eq:div_dofs_omega}. We also notice that from \eqref{eq:rotation} we deduce that the linear operator
			$$
			\bfomega_E \longmapsto \int_E  \big(\bbx -\bbx_E\big)\wedge \left[\bfomega_E\wedge
			\big(\bbx -\bbx_E\big)\right]\dEl
			$$
			is symmetric and positive definite.	
		\end{proof}
		\noindent From Proposition \ref{pr:divFromdofs} and \eqref{eq:face_approx} we infer that the dimension of the space \eqref{eq:local_stress} is $$\dim( \Sigma_h(E))=6\,n^E_f.$$
		%
		%
		\paragraphNew{Displacement space.}
		The local approximation space for the displacement field is defined by, see \eqref{eq:rigid}: 
		\begin{equation}\label{eq:local_displ}
		U_h(E)=\left\{ \bbv_h\in  \left[L^2(E)\right]^3\ :\ \bbv_h\in RM(E) \right\},
		\end{equation}
		and it follows that  $$\dim(U_h(E))=6.$$
		
		%
		%
		
		%
		%
		\subsection{The local forms}\label{ss:E-bforms}
		In this section we introduce the VEM counterparts of the local forms associated with the continuous problem.
		\paragraphNew{The local mixed term.}Given $E\in \Th$, we begin by noticing that the term
		\begin{equation*}
		\left(\bdiv \bftau_h,\bbv_h\right)_{E}=\int_E \bdiv \bftau_h\cdot \bbv_h~\dEl
		\end{equation*}
		is computable for every $\bftau_h\in\Sigma_h(E)$ and $\bbv_h\in U_h(E)$ via degrees of freedom. For this reason, we do not need to introduce any approximation of the terms $(\bdiv \bftau, \bbu)$ and $(\bdiv \bfsigma, \bbv)$ in problem~\eqref{cont-pbl}.
		\paragraphNew{The local bilinear form $a_{E}(\cdot,\cdot)$.}
		The local bilinear form
		\begin{equation*}
		a_E(\bfsigma_h,\bftau_h)  = \int_E \D \bfsigma_h : \bftau_h~\dEl
		\end{equation*}
		is not computable for a general couple $(\bfsigma_h,\bftau_h)\in \Sigma_h(E)\times \Sigma_h(E)$.
		As it is standard in the VEM procedure (cf. \cite{volley}), we build a computable approximation of the bilinear form by defining a suitable projection operator onto local polynomial functions. In our case, we introduce $\Pi_E : \Sigma_h(E)\to \left[\P_0(E)\right]^{3\times 3}_s$, by requiring 
		\begin{equation}\label{eq:proj}
		\int_E \Pi_E\,\bftau_h: \bfpi_0 = \int_E \bftau_h: \bfpi_0~\dEl \qquad \forall\bfpi_0 \in \left[\P_0(E)\right]^{3\times 3}_s .
		\end{equation}
		This is a projection operator onto the constant symmetric tensor functions and it is computable. 
		Indeed, we notice that each $\bfpi_0 \in \left[\P_0(E)\right]^{3\times 3}_s$ can be written as the symmetric gradient of a linear vectorial function, i.e. $\bfpi_0 = \teps(\bbp_1)$, with $\bbp_1\in \left[\P_1(E)\right]^{3}$.  
		Hence,
		using the divergence theorem, the right-hand side of \eqref{eq:proj} becomes 
		\begin{equation*}
		\int_E \bftau_h: \bfpi_0~\dEl= \int_E \bftau_h: \teps(\bbp_1)~\dEl = -\int_E \bdiv \bftau_h \cdot \bbp_1~\dEl+\int_{ \partial E}(\bftau_h \bbn) \cdot \bbp_1~\df
		\end{equation*}
		which is clearly computable (see also Proposition \ref{pr:divFromdofs}).
		Then, the approximation of $a_E(\cdot,\cdot)$ reads: 
		\begin{equation}\label{eq:ah1}
		\begin{aligned}
		a_E^h(\bfsigma_h,\bftau_h)  &:=
		a_E(\Pi_E\,\bfsigma_h,\Pi_E\bftau_h) + s_E\left( (I-\Pi_E)\bfsigma_h, (I-\Pi_E)\bftau_h \right)\\
		&=\int_E \D (\Pi_E\bfsigma_h) : (\Pi_E\bftau_h)~\dEl + s_E\left( (I-\Pi_E)\bfsigma_h, (I-\Pi_E)\bftau_h \right) ,
		\end{aligned}
		\end{equation}
		where $s_E(\cdot,\cdot)$ is a suitable stabilization term. In this paper we propose the following choice:
		\begin{equation}\label{eq:stab1}
		s_E(\bfsigma_h,\bftau_h) : = \kappa_E\, h_E\int_{\partial E} (\bfsigma_h\bbn)\cdot \bftau_h\bbn~\df,
		\end{equation}
		Above, $\kappa_E$ is a positive constant to be chosen according to $\D$. For instance, in the numerical examples of Section \ref{s:numer}, $\kappa_E$ is set equal to $\frac{1}{2} {\rm tr}(\D_{|E})$. However, any norm of $\D_{|E}$ can be used. A possible variant of \eqref{eq:stab1} is
		\begin{equation*}
		s_E(\bfsigma_h,\bftau_h) : = \kappa_E\, \sum_{f\in\partial E}h_f\int_{f} (\bfsigma_h\bbn)\cdot \bftau_h\bbn~\df.
		\end{equation*}
		\paragraphNew{The local loading term.} We split the load term on each element and we have  
		\begin{equation*}
		(\bbf,\bbv_h)=\int_{\Omega}\bbf\cdot\bbv_h~\dO =\sum_{E\in\Th}\int_{E}\bbf\cdot\bbv_h~\dEl.
		\end{equation*}
		Since $\bbv_h\in RM(E)$, the right-hand side is computable via quadrature rules for polyhedral domains.
		
		\begin{remark}
			Since this integral involves a sufficiently regular function $f$,
			to get a ``good'' approximation we exploit a quadrature rule of high degree. 
			For the numerical examples in Section~\ref{s:numer}, we use a quadrature rule of degree~4.
	\end{remark}}

	\subsection{The discrete scheme}\label{ss:discrete}
	
	Starting from the local spaces and local terms introduced in the previous sections, 
	we can set the global problem. More specifically, we introduce a global approximation space for the stress field, by glueing the local approximation spaces, see \eqref{eq:local_stress}:
	
	\begin{equation}\label{eq:global-stress}
	\Sigma_h=\Big\{ \bftau_h\in  H(\bdiv;\Omega)\ :\ \bftau_{h|E}\in  \Sigma_h(E)\quad \forall E\in\Th \Big\}.
	\end{equation}
	For the global approximation of the displacement field, we take, see \eqref{eq:local_displ}:
	
	\begin{equation}\label{eq:global-displ}
	U_h=\Big\{ \bbv_h\in  \left[L^2(\Omega)\right]^3\ :\ \bbv_{h|E}\in U_h(E)\quad \forall E\in\Th \Big\}.
	\end{equation}
	Then, given a local approximation of $a_E(\cdot,\cdot)$, see \eqref{eq:ah1}, we set
	
	\begin{equation}\label{eq:global-ah}
	a_h(\bfsigma_h,\bftau_h):= \sum_{E\in\Th}a_E^h(\bfsigma_h,\bftau_h) .
	\end{equation}
	The method we consider is then defined by
	
	\begin{equation}\label{eq:discr-pbl-ls}
	\left\lbrace{
		\begin{aligned}
		&\mbox{Find } (\bfsigma_h,\bbu_h)\in \Sigma_h\times U_h~\mbox{such that}\\
		&a_h(\bfsigma_h,\bftau_h)    + (\bdiv \bftau_h, \bbu_h)= 0 \quad & \forall \bftau_h \in \Sigma_h\\
		& (\bdiv \bfsigma_h, \bbv_h) = -(\bbf,\bbv_h)\quad &\forall \bbv_h\in U_h .
		\end{aligned}
	} \right.
	\end{equation}
	
	\section{Stability and convergence analysis}\label{s:theoretical}
	
	Since some results of the analysis follows the guidelines of the theory developed in \cite{ARTIOLI2017155}, in this section we do not provide full details of all the proofs.
	First, for all $E\in \Th$, we introduce the space:
	
	\begin{equation*}
	\widetilde\Sigma(E):=\left\{ \bftau\in H(\bdiv;E)\ : \ \exists \bbw\in \left[H^1(E)\right]^3 \mbox{ \rm such that } \bftau=\C\teps(\bbw) \right\} .
	\end{equation*}
	The global space $\widetilde\Sigma$ is defined as
	\begin{equation*}
	\widetilde\Sigma:=\left\{ \bftau\in H(\bdiv;\Omega)\ : \ \exists \bbw\in  \left[H^1(\O) \right]^3 \mbox{ \rm such that } \bftau=\C\teps(\bbw) \right\} .
	\end{equation*}
	In the sequel, given a measurable subset $A\subseteq \Omega$ and $r > 2$, we will use the following space
	\begin{equation*}
	W^r(A):=\left\{  \bftau  \ : \bftau\in \left[L^r(A)\right]^{3\times 3}_s \ , \   \bdiv\bftau\in \left[L^2(A)\right]^3    \right\} ,
	\end{equation*}
	equipped with the obvious norm.
	%
	%
	\subsection{An interpolation operator for stresses}\label{ss:interpol-oper}
	We introduce the local interpolation operator $\IE : W^r(E)\to  \Sigma_h(E)$, defined by:
	\begin{equation}\label{eq:loc-interp_def}
	\int_{\partial E} (\IE \bftau) \bbn\cdot \bfvarphi_\ast~\df = \int_{\partial E}  (\bftau\,\bbn)\cdot \bfvarphi_\ast~\df \qquad \forall \bfvarphi_\ast\in  R_\ast(\partial E),
	\end{equation}
	where: 
	\begin{equation}\label{eq:Rast}
	\begin{aligned}
	R_\ast(\partial E) = \Big\{
	\bfvarphi_\ast\in\left[ L^2(\partial E)\right]^3 \,: \,
	\bfvarphi_{\ast | f}(\bxt) = \bfgamma_f + & \left[\bfdelta_f\wedge(\bbx(\bxt) -\bbx_E)\right], \\
	& \bfgamma_f,\,\bfdelta_f \in\R^3,  \ \forall f\in\partial E  \Big\}.
	\end{aligned}
	\end{equation}
	We remark that if $\bftau$ is not sufficiently regular, the integral in the right-hand side of \eqref{eq:loc-interp_def} is intended as a duality between $\left[W^{-\frac{1}{r},r}(\partial E)\right]^3$ and $\left[W^{\frac{1}{r},r'}(\partial E)\right]^3$. Instead, if $\bftau$ is a regular function, the above condition is equivalent to require:
	\begin{equation}\label{eq:loc-interp}
	\left\lbrace{
		\begin{aligned}
		&\int_f (\IE \bftau) \bbn\cdot \bfalpha~\df= \int_f  (\bftau\,\bbn) \cdot \bfalpha~\df \quad & \forall \bfalpha\in \R^3;\\
		&\int_f (\IE \bftau) \bbn\cdot \left[\bfomega\wedge(\bbx(\bxt) - \bbx_E)\right]\df = \int_f  (\bftau\,\bbn) \cdot \left[\bfomega\wedge(\bbx(\bxt) - \bbx_E)\right]\df\quad& \forall \bfomega\in \R^3 ;\\
		\end{aligned}
	} \right.
	\end{equation}
	for each face $f\in\partial E$. 
	
	We now show that $\IE\bftau\in\Sigma_h(E)$ is well-defined by conditions~\eqref{eq:loc-interp_def}. Indeed this is an immediate consequence of the following Lemma.
	
	\begin{lem}\label{lm:unisolve}
		If $\bftau_h\in \Sigma_h(E)$ is such that 
		
		\begin{equation*}
		\int_{\partial E}  (\bftau_h\,\bbn)\cdot \bfvarphi_\ast~\df = 0 \qquad \forall \bfvarphi_\ast\in  R_\ast(\partial E),
		\end{equation*}
		then $\bftau_h = \bfzero$.
	\end{lem}
	
	\begin{proof}
		Recalling \eqref{eq:local_stress}, Proposition \ref{pr:divFromdofs} and \eqref{eq:Rast}, it is sufficient to prove that, given a face $f\in\partial E$, conditions
		\begin{equation}
		\label{eq:unisolve2}
		\int_{f}  (\bftau_h\,\bbn)\cdot \left[\bfgamma_f +  \bfdelta_f\wedge(\bbx(\bxt) -\bbx_E) \right]\df = 0 \qquad \forall \bfgamma_f,\,\bfdelta_f \in\R^3 
		\end{equation}
		imply $(\bftau_h\,\bbn)_{|f} = \bfzero$. To this end, we first set (cf. \eqref{eq:face_approx})
		
		\begin{equation}\label{eq:unisove3}
		(\bftau_h\,\bbn)_{|f}(\bxt) =\,\bbt_{f} + a\big[\bbn_{f}\wedge(\bbx(\bxt) -\bbx_f)\big]+ p_1(\bxt)\bbn_{f},
		\end{equation}
		with $\bbt_{f}$ an arbitrary constant vector tangent to the face $f$, $a\in\R$, $\bbn_{f}$ the outward normal vector and 
		$$
		p_1(\bxt)=c_1+c_2(\xt-\xt_f)+c_3(\yt-\yt_f)\qquad c_i\in\R\quad i=1,2,3.
		$$  
		Choosing $\bfdelta_f=\bfzero$ and $\bfgamma_f$ arbitrary in \eqref{eq:unisolve2}, and considering \eqref{eq:unisove3}, we infer $\bbt_f=\bfzero$ and $c_1=0$. Hence, it holds:
		\begin{equation}\label{eq:unisove4}
		(\bftau_h\,\bbn)_{|f}(\bxt) =\, a\big[\bbn_{f}\wedge(\bbx(\bxt) -\bbx_f)\big]+ \left(c_2(\xt-\xt_f)+c_3(\yt-\yt_f)\right)\bbn_{f}.
		\end{equation}
		We now select $\bfgamma_f=\bfzero$ and $\bfdelta_f=\bbn_f$ in \eqref{eq:unisolve2}. From \eqref{eq:unisove4} we get
		\begin{equation*}
		a\, \int_{f}  \big[\bbn_{f}\wedge(\bbx(\bxt) -\bbx_f)\big]\cdot \big[ \bbn_f\wedge(\bbx(\bxt) -\bbx_E) \big]\df = 0 .
		\end{equation*}
		We have:
		\begin{equation}
		\label{eq:unisolve6}
		\begin{aligned}
		&\int_{f}  \big[\bbn_{f}\wedge(\bbx(\bxt) -\bbx_f)\big]\cdot \big[ \bbn_f\wedge(\bbx(\bxt) -\bbx_E) \big]\df =\\
		&\int_{f}  \big[\bbn_{f}\wedge(\bbx(\bxt) -\bbx_f)\big]\cdot \big[ \bbn_f\wedge((\bbx(\bxt) - \bbx_f) + (\bbx_f-\bbx_E)) \big]\df =\\
		& \int_{f}  \big|\bbn_{f}\wedge(\bbx(\bxt) -\bbx_f)\big|^2\df + \int_{f} \big[\bbn_{f}\wedge(\bbx(\bxt) -\bbx_f)\big]\cdot \big[ \bbn_f\wedge (\bbx_f-\bbx_E) \big]\df =\\
		&\int_{f}  \big|\bbn_{f}\wedge(\bbx(\bxt) -\bbx_f)\big|^2\df > 0 .
		\end{aligned}
		\end{equation}
		Above, we have used that 
		$\bbn_f\wedge (\bbx_f-\bbx_E)$ is a constant vector and that $\bbn_{f}\wedge(\bbx(\bxt) -\bbx_f)$ has zero mean value over the face $f$, to infer that:
		$$
		\int_{f} \big[\bbn_{f}\wedge(\bbx(\bxt) -\bbx_f)\big]\cdot \big[ \bbn_f\wedge (\bbx_f-\bbx_E) \big]\df = 0
		$$
		From \eqref{eq:unisolve6} we deduce $a=0$, and therefore we get:
		\begin{equation}\label{eq:unisolve7}
		(\bftau_h\,\bbn)_{|f}(\bxt) =\,  \left(c_2(\xt-\xt_f)+c_3(\yt-\yt_f)\right)\bbn_{f}.
		\end{equation}
		We finally select $\bfgamma_f=\bfzero$ and $\bfdelta_f=\bbt_f$ in \eqref{eq:unisolve2}, with $\bbt_f$ an arbitrary vector tangential to the face $f$. Thus, we obtain:
		\begin{equation*}
		\int_f  \left(c_2(\xt-\xt_f)+c_3(\yt-\yt_f)\right)\bbn_{f}\cdot\left[\bbt_f \wedge (\bbx(\bxt) - \bbx_E) \right] \df =0 .
		\end{equation*}
		Using again $\bbx(\bxt) - \bbx_E = (\bbx(\bxt) - \bbx_f) + (\bbx_f-\bbx_E)$, we infer
		\begin{equation}\label{eq:unisolve9}
		\begin{aligned}
		&T := \int_f  \left(c_2(\xt-\xt_f)+c_3(\yt-\yt_f)\right)\bbn_{f}\cdot\left[\bbt_f \wedge (\bbx(\bxt) - \bbx_f) \right] \df  =\\ 
		&\int_f \left(c_2(\xt-\xt_f)+c_3(\yt-\yt_f)\right) (\bbx(\bxt) - \bbx_f)\cdot\left[\bbn_{f} \wedge \bbt_f \right] \df = 0 .
		\end{aligned}
		\end{equation}
		We now choose $\bbt_f$ such that the (tangential to the face) vector $\bbn_{f} \wedge \bbt_f$ has components $(c_2,c_3)$ with respect to the local coordinate system $(\xt, \yt)$. Then, from \eqref{eq:unisolve9} we get 
		\begin{equation*}
		T = \int_f \left(c_2^2(\xt-\xt_f)^2+c_3^2(\yt-\yt_f)^2\right) \df = 0 ,
		\end{equation*}
		which implies $c_2=c_3=0$. Therefore, $(\bftau_h\,\bbn)_{|f} = \bfzero$, see \eqref{eq:unisolve7}, and the proof is complete.

	\end{proof}

	The global interpolation operator $\Ih: W^r(\Omega)\to\Sigma_h$ is then defined by glueing the local contributions provided by $\IE$. More precisely, for every $\bftau\in W^r(\O)$ and $E\in\T_h$, we set
	$$(\Ih\bftau)_{|E}:=\IE\bftau_{|E}.$$
	Using the same steps detailed in~\cite{ARTIOLI2017155}, we have the following error estimates for the interpolation operator $\Ih$.
	\begin{prop}\label{pr:approxest} Under assumptions $\mathbf{(A1)}$, $\mathbf{(A2)}$ and $\mathbf{(A3)}$,
		for the interpolation operator $\IE$ defined in \eqref{eq:loc-interp}, the following estimates hold:
		
		\begin{equation*}
		|| \bftau -\IE\bftau||_{0,E}\lesssim h_E |\bftau|_{1,E} \qquad \forall \bftau\in  \widetilde\Sigma(E) \cap \left[H^1(E)\right]^{3\times 3}\coorF{,} 
		\end{equation*}
		and
		\begin{equation*}
		\begin{aligned}
		|| \bdiv(\bftau -\IE\bftau)||_{0,E}  \lesssim h_E |\bdiv\bftau|_{1,E}   \ \
		\forall \bftau\in \widetilde\Sigma(E) \cap \left[H^1(E)\right]^{3\times 3} 
		\mbox{ \rm s.t.  $\bdiv\bftau\in \left[H^1(E)\right]^3$}.
		\end{aligned}
		\end{equation*}	
	\end{prop}
	
	%
	%
	
	\subsection{The {\em ellipticity-on-the-kernel} and {\em inf-sup} conditions}\label{ss:elker_infsup}
	
	The proposed approach satisfies the compatibility conditions.
	First, we notice that (see \eqref{eq:global-stress}, \eqref{eq:local_stress} and \eqref{eq:global-displ}, \eqref{eq:local_displ}):
	
	\begin{equation}\label{eq:kern-incl}
	\bdiv(\Sigma_h)\subseteq U_h .
	\end{equation}
	Then, we introduce the discrete kernel $K_h\subseteq \Sigma_h$:
	\begin{equation*}
	K_h =\{ \bftau_h\in\Sigma_h\, :\, (\bdiv \bftau_h,\bbv_h)=0 \quad \forall \bbv_h\in U_h  \},
	\end{equation*}
	and we infer from \eqref{eq:kern-incl} that $\bftau_h\in K_h$ implies $\bdiv \bftau_h=\bfzero$. Hence, it holds:
	
	\begin{equation}\label{eq:l2-hdiv}
	|| \bftau_h ||_\Sigma = ||\bftau_h ||_0\qquad \forall \bftau_h\in K_h .
	\end{equation}
	This is essentially the property that leads to the following {\em ellipticity-on-the-kernel} condition.
	
	\begin{prop}\label{pr:elker}
		For the method described in Section \ref{s:HR-VEM}, there exists a constant $\alpha >0$ such that
		
		\begin{equation}\label{eq:elker}
		a_h(\bftau_h,\bftau_h)\ge \alpha\, || \bftau_h||^2_\Sigma\qquad \forall \bftau_h\in K_h .
		\end{equation}
		
	\end{prop}
	
	\begin{proof}
		Fix $E\in\Th$. By \eqref{eq:proj}, \eqref{eq:ah1} and \eqref{eq:stab1}, using the techniques of \cite{volley,BFMXX}, one has:
		
		\begin{equation*}
		||\bftau_h||_{0,E}^2\lesssim a_E^h(\bftau_h,\bftau_h)\lesssim ||\bftau_h||_{0,E}^2\qquad \forall \bftau_h\in\Sigma_h(E).
		\end{equation*}
		By recalling \eqref{eq:global-ah}, we get the existence of $\alpha>0$ such that
		
		\begin{equation*}
		a_h(\bftau_h,\bftau_h)\ge \alpha\, || \bftau_h||^2_0\qquad \forall \bftau_h\in \Sigma_h .
		\end{equation*}
		Estimate \eqref{eq:elker} now follows by recalling \eqref{eq:l2-hdiv}.
	\end{proof}
	For the discrete {\em inf-sup} condition, we need the following {\em commuting diagram property}.
	\begin{prop}\label{pr:diagramCommuatative}
		For the operator $\Ih: W^r(\O)\rightarrow\Sigma_h$ it holds:
		
		\begin{equation}\label{eq:diagram_commutative}
		\bdiv (\Ih\bftau)= \Pi_{RM}( \bdiv \bftau) \qquad \forall \bftau \in W^r(\O),
		\end{equation}
		where $\Pi_{RM}$ denotes the $L^2$-projection operator onto the space of the rigid body motions, see \eqref{eq:rigid}.
	\end{prop}
	%
	%
	\begin{proof}
		It is sufficient to prove property \eqref{eq:diagram_commutative} locally, in each element $E \in\T_h$. 
		Fix now $\bbr\in RM(E)$ and $\bftau\in W^r(E)$. We have:
		\begin{equation}\label{eq:diagram1}
		\begin{aligned}
		\int_E \bdiv \bftau \cdot \bbr~\dEl &= \int_{\partial E}( \bftau\bbn ) \cdot \bbr~\df \quad &(\text{by \eqref{eq:loc-interp}})\\
		&= \int_{\partial E} (\IE\bftau)\bbn  \cdot \bbr~\df \quad& (\text{integration by parts})\\
		&= \int_E \bdiv (\IE\bftau) \cdot \bbr~\dEl 
		\end{aligned}
		\end{equation}	
		From \eqref{eq:diagram1} and the definition of $L^2$-projection operator, we get $\bdiv(\IE\bftau)=\Pi_{RM}(\bdiv \bftau)$ on $E$.
	\end{proof}
	Using Proposition \ref{pr:diagramCommuatative} the following discrete {\em inf-sup} condition follows from the theory developed in~\cite{ARTIOLI2017155}.
	\begin{prop}\label{pr:inf-sup} Suppose that  assumptions $\mathbf{(A1)}$, $\mathbf{(A2)}$ and $\mathbf{(A3)}$ are fulfilled. Then, there exists $\beta>0$ such that
		
		\begin{equation*}
		\sup_{\bftau_h\in \Sigma_h}\frac{(\bdiv \bftau_h,\bbv_h)}{|| \bftau_h||_{\Sigma}}\ge \beta ||\bbv_h||_{U}\qquad \forall\, \bbv_h\in U_h .
		\end{equation*}
		
	\end{prop}
	
	\subsection{Error estimates}\label{ss:errest}
	We denote with $\P_0(\Th)$ the space of piecewise constant functions with respect to the given mesh $\Th$. Using the techniques developed in \cite{ARTIOLI2017155}, one can prove the following result.
	\begin{prop}\label{pr:error-est} Suppose that  assumptions $\mathbf{(A1)}$, $\mathbf{(A2)}$ and $\mathbf{(A3)}$ are fulfilled.
		For every $(\bfsigma_I,\bbu_I)\in\Sigma_h\times U_h$ and every $\bfsigma_\pi\in \left[\P_0(\Th)\right]^{3\times 3}_s$, the following error equation holds:
		\begin{equation*}
		|| \bfsigma- \bfsigma_h||_\Sigma + || \bbu - \bbu_h ||_U \lesssim || \bfsigma- \bfsigma_I||_\Sigma + || \bbu - \bbu_I ||_U  + h\,||\bdiv\bfsigma_I||_{0,\O} + || \bfsigma- \bfsigma_\pi||_{0,\O}
		.
		\end{equation*}
	\end{prop}
	A suitable choice of $\bbu_I$, $\bfsigma_I$, and $\bfsigma_\pi$ leads to the following error estimate, see \cite{ARTIOLI2017155}.
	\begin{thm}\label{th:main_convergence}
		Let $(\bfsigma,\bbu)\in\Sigma\times U$ be the solution of Problem \eqref{cont-pbl}, and let $(\bfsigma_h,\bbu_h)\in\Sigma_h\times U_h$ be the solution of the discrete problem \eqref{eq:discr-pbl-ls}. Suppose that  assumptions $\mathbf{(A1)}$, $\mathbf{(A2)}$ and $\mathbf{(A3)}$ are fulfilled.
		Assuming $\bfsigma_{|E}\in \left[H^1(E)\right]^{3\times 3}$ and $(\bdiv\, \bfsigma)_{|E}\in \left[H^1(E)\right]^3$, the following estimate holds true:
		\begin{equation*}
		|| \bfsigma - \bfsigma_h||_{\Sigma} + || \bbu - \bbu_h||_U \lesssim C(\Omega,\bfsigma,\bbu)\, h ,
		\end{equation*}
		where $C(\Omega,\bfsigma,\bbu)$ is independent of $h$ but depends on the domain $\Omega$ and on the Sobolev regularity of $\bfsigma$ and $\bbu$.
	\end{thm}
	\section{Numerical results}\label{s:numer}
	In this section we numerically assess the proposed VEM approach through the study of the method accuracy on a selected number of test problems. The numerical results confirm the proved theoretical results.
	We consider the standard unit cube $\O=[0,1]^3$ as the domain of our problems and we take the following three types of mesh: 
	
	\begin{itemize}
		\item \textbf{Cube,} a mesh composed by standard structured cubes;
		\item \textbf{Tetra,} a Delaunay tetrahedralization of the domain $\O$;
		\item \textbf{CVT,} a Voronoi tassellation obtained by the Lloyd algorithm~\cite{CVT1999};
		\item \textbf{Random,} a Voronoi tassellation achieved with random control points.
	\end{itemize}
	
	We remark that the meshes \textbf{CVT} and \textbf{Random} are very challenging. Indeed, they could have some elements with small faces and edges,
	and we remark that such case is \emph{not} covered by the developed theory, i.e. assumptions $\mathbf{(A1)}$, $\mathbf{(A2)}$ and $\mathbf{(A3)}$.
	However, the numerical results show that the proposed methods are fairly robust with respect to this geometric situation.
	These two type of meshes are build via the \texttt{voro++} library~\cite{Voro++}. 
	Moreover, the whole numerical scheme is developed inside the \texttt{vem++} library, a \texttt{c++} code realized at the Univeristy Milano - Bicocca during the 
	CAVE project ({\tt https://sites.google.com/view/vembic/home}).
	
	%
	%
	%
	%
	\newcommand{\sizeGraph}{0.49}
	\newcommand{\sizeMesh}{0.47}
	\begin{figure}[h!]
		\centering
		\renewcommand{\thesubfigure}{}
		\subfigure[Cube]{\includegraphics[width=\sizeMesh\textwidth,trim = 0mm 0mm 0mm 0mm, clip]{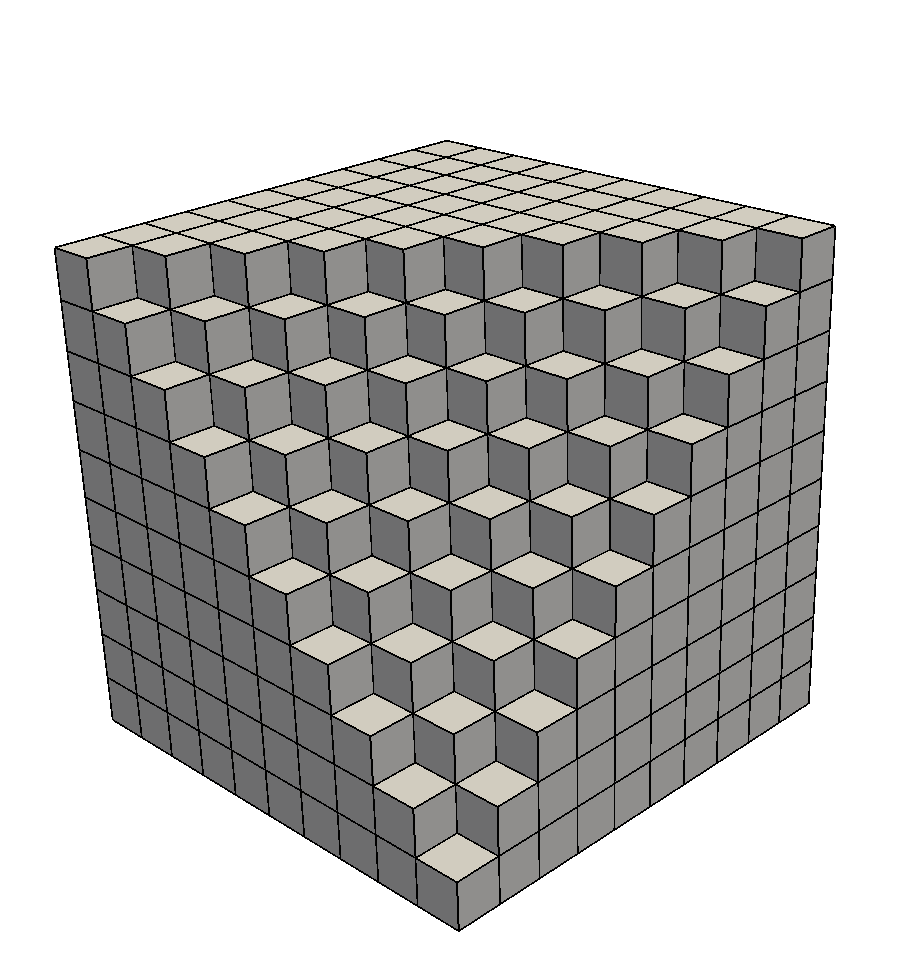}}\qquad
		\subfigure[Tetra]{\includegraphics[width=\sizeMesh\textwidth,trim = 0mm 0mm 0mm 0mm, clip]{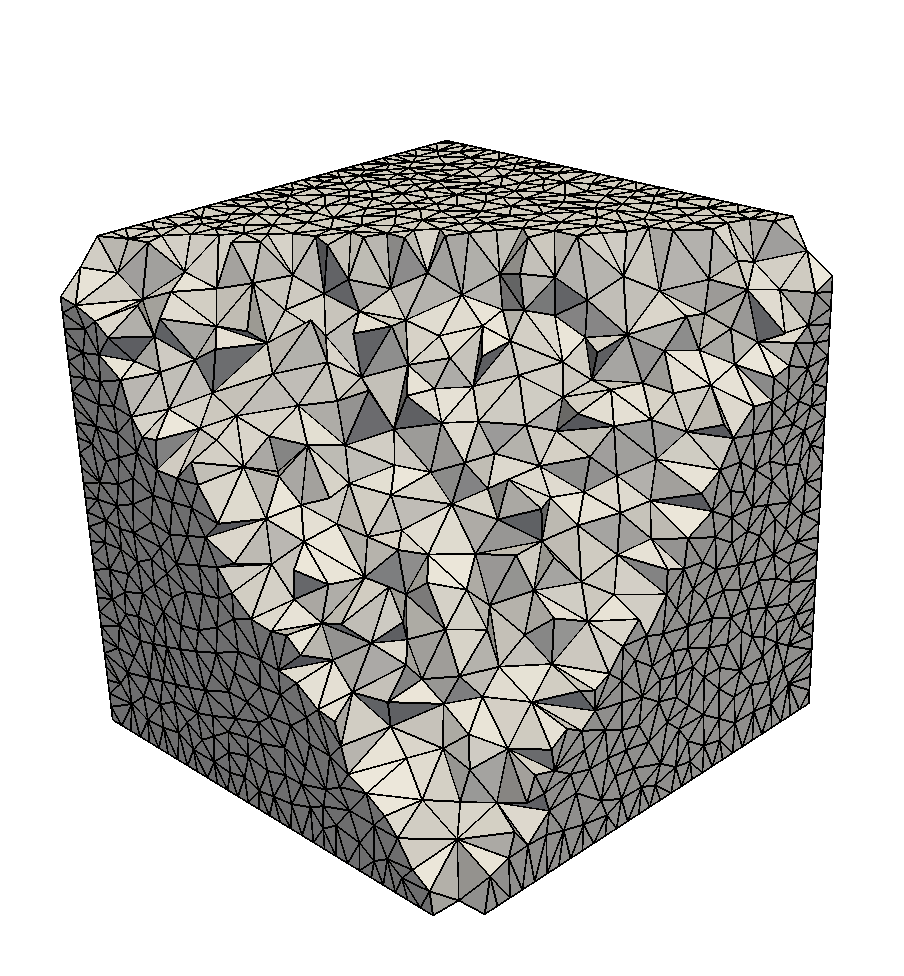}}\\
		\subfigure[CVT]{\includegraphics[width=\sizeMesh\textwidth,trim = 0mm 0mm 0mm 0mm, clip]{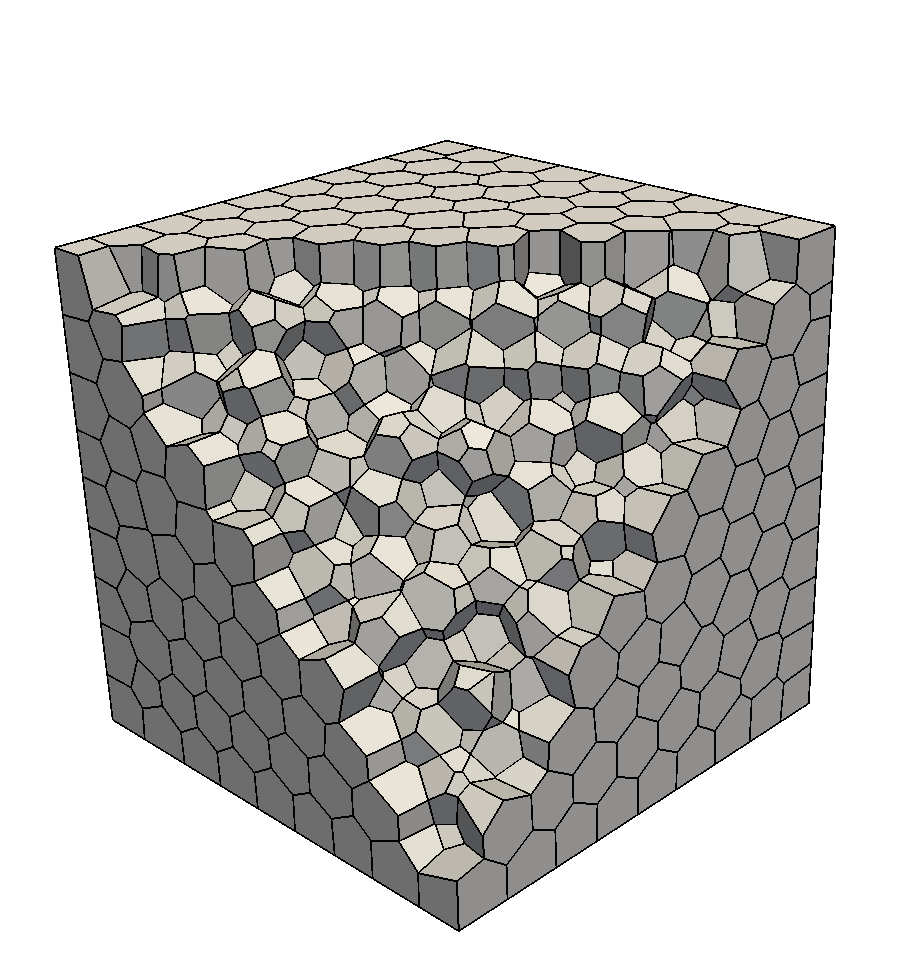}}\qquad
		\subfigure[Random]{\includegraphics[width=\sizeMesh\textwidth,trim = 0mm 0mm 0mm 0mm, clip]{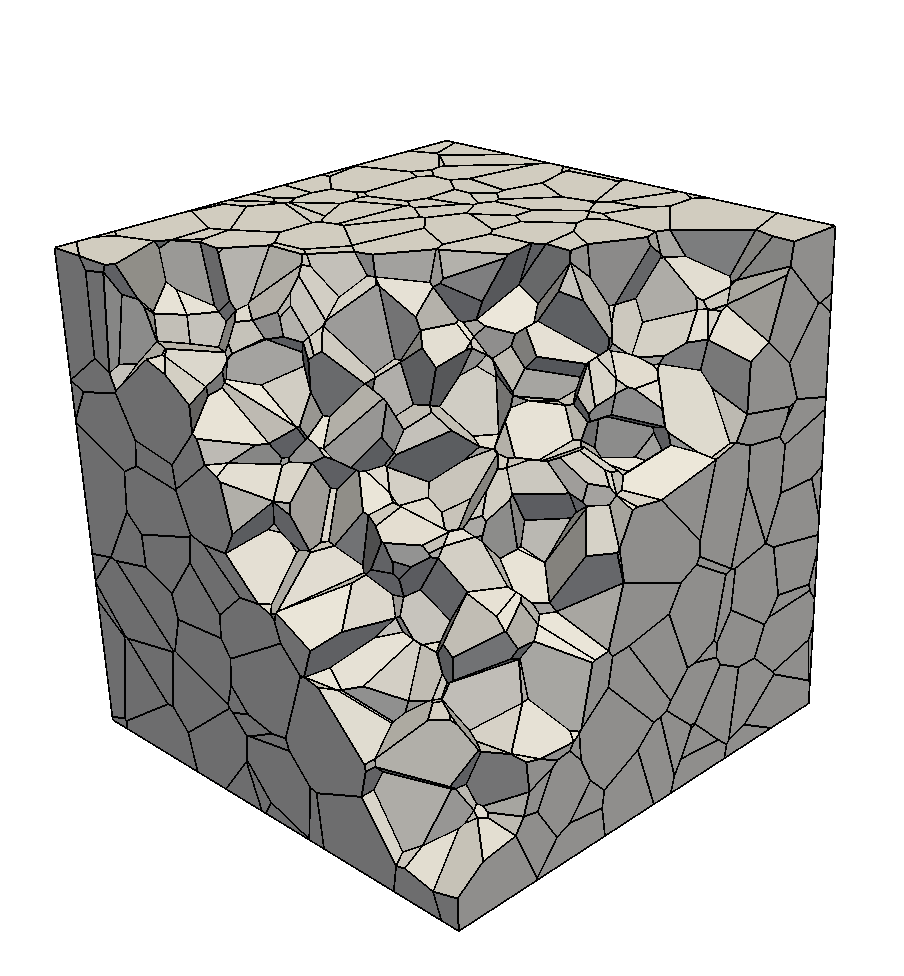}}
		\caption{Overview of adopted meshes for convergence assessment numerical tests.}
		\label{fig:meshes}
	\end{figure}
	In order to assess the convergence rate, for each type of mesh, we define the following mesh-size $h$:
	\begin{equation*}
	h:=\dfrac{1}{N_E}\sum_{i=1}^{N_E} h_E
	\end{equation*}
	where we recall that $N_E$ is the number of elements in the mesh, and $h_E$ is the diameter of the polyhedron $E$.
	%
	%
	The accuracy and the convergence rate assessment is carried out using the following error norms:
	\begin{itemize}
		
		\item[$\bullet$] $L^2$ error norm for the displacement field:
		\begin{equation*}
		E_{\bbu} :=\left( \sum_{E \in \Th} \int_{E}  | \bbu - \bbu_h |^2 \right)^{1/2}= ||\bbu - \bbu_h ||_0.
		\end{equation*}
		
		\item[$\bullet$] $L^2$ error on the divergence:
		\begin{equation*}
		E_{\bfsigma, \bdiv}   :=\left( \sum_{E \in \Th} \int_E | \bdiv(\bfsigma - \bfsigma_h)  |^2\right)^{1/2} .
		\end{equation*}
		
		\item[$\bullet$] $L^2$ error on the projection:
		\begin{equation*}
		E_{\bfsigma, \Pi}   :=\left( \sum_{E \in \Th} \int_E |\bfsigma - \Pi_E\bfsigma_h |^2\right)^{1/2} .
		\end{equation*}	
		
		\item[$\bullet$] Discrete error norms for the stress field:
		\begin{equation*}
		E_{\bfsigma}   :=\left( \sum_{f \in \Fh} h_f\int_{f} \kappa\,| (\bfsigma - \bfsigma_h)\bbn  |^2\right)^{1/2} ,
		\end{equation*}
		where $ \Fh$ is the set of faces for $\T_h$ and $\kappa=\frac{1}{2} {\rm tr}(\D)$ (the material is here homogeneous).
		We remark that the quantity above scales like the internal elastic energy, with respect to the size of the domain and of the elastic coefficients, i.e. $\sim h$.
	\end{itemize}
	%
	%
	%

	%
	%
	\paragraph{Example 1 (compressible material).}\label{par:Comprimibile}
	We consider an elastic problem with a trigonometric solution and homogeneous Dirichlet boundary conditions. The material of this problems obeys to a homogeneous isotropic constituive law, with material parameters assigned in terms of the Lam\'e constants, here set as $\lambda = 1$ and $\mu = 1$. Applied loads are accordingly computed. More precisely, the test details are as follows:
	\begin{eqnarray*}
	\left\{
	\begin{array}{l}
	u_1 = u_2 = u_3 = 10\, S(x,y,z) \\
	f1= -10\pi^2((\lambda+\mu)\cos(\pi x)\sin(\pi y +\pi z)- (\lambda + 4 \mu)S(x,y,z))\\
	f2= -10\pi^2((\lambda+\mu)\cos(\pi y)\sin(\pi x +\pi z)- (\lambda + 4 \mu)S(x,y,z))\\
	f3= -10\pi^2((\lambda+\mu)\cos(\pi z)\sin(\pi x +\pi y)- (\lambda + 4 \mu)S(x,y,z)),
	\end{array}
	\right .
	\end{eqnarray*}
	where $S(x,y,z)=\sin(\pi x)\sin(\pi y)\sin(\pi z)$.\\
	\begin{figure}[htb!]
		\centering
		\subfigure[]{\includegraphics[width=\sizeGraph\textwidth,trim = 0mm 0mm 0mm 0mm, clip]{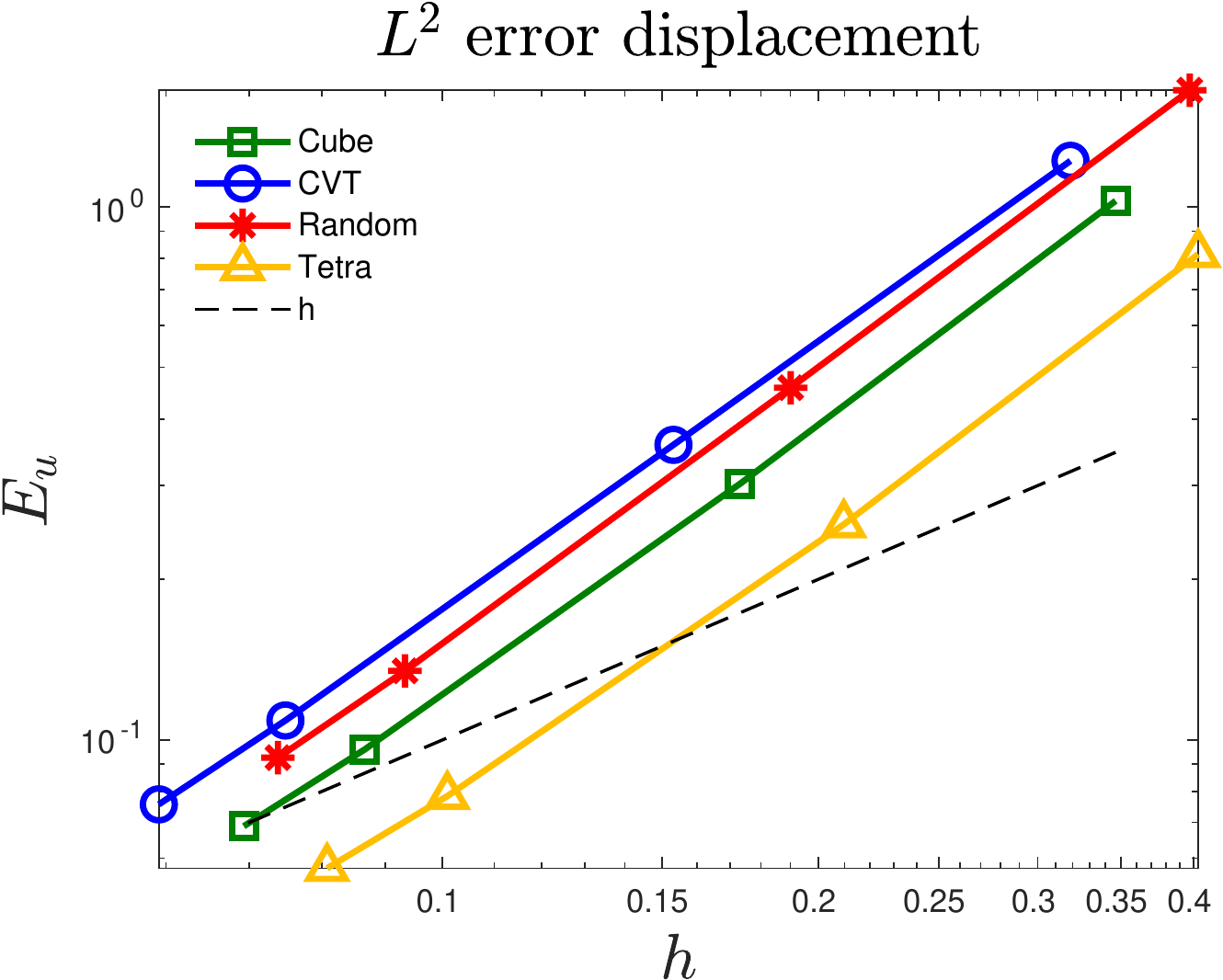}}
		\subfigure[]{\includegraphics[width=\sizeGraph\textwidth,trim = 0mm 0mm 0mm 0mm, clip]{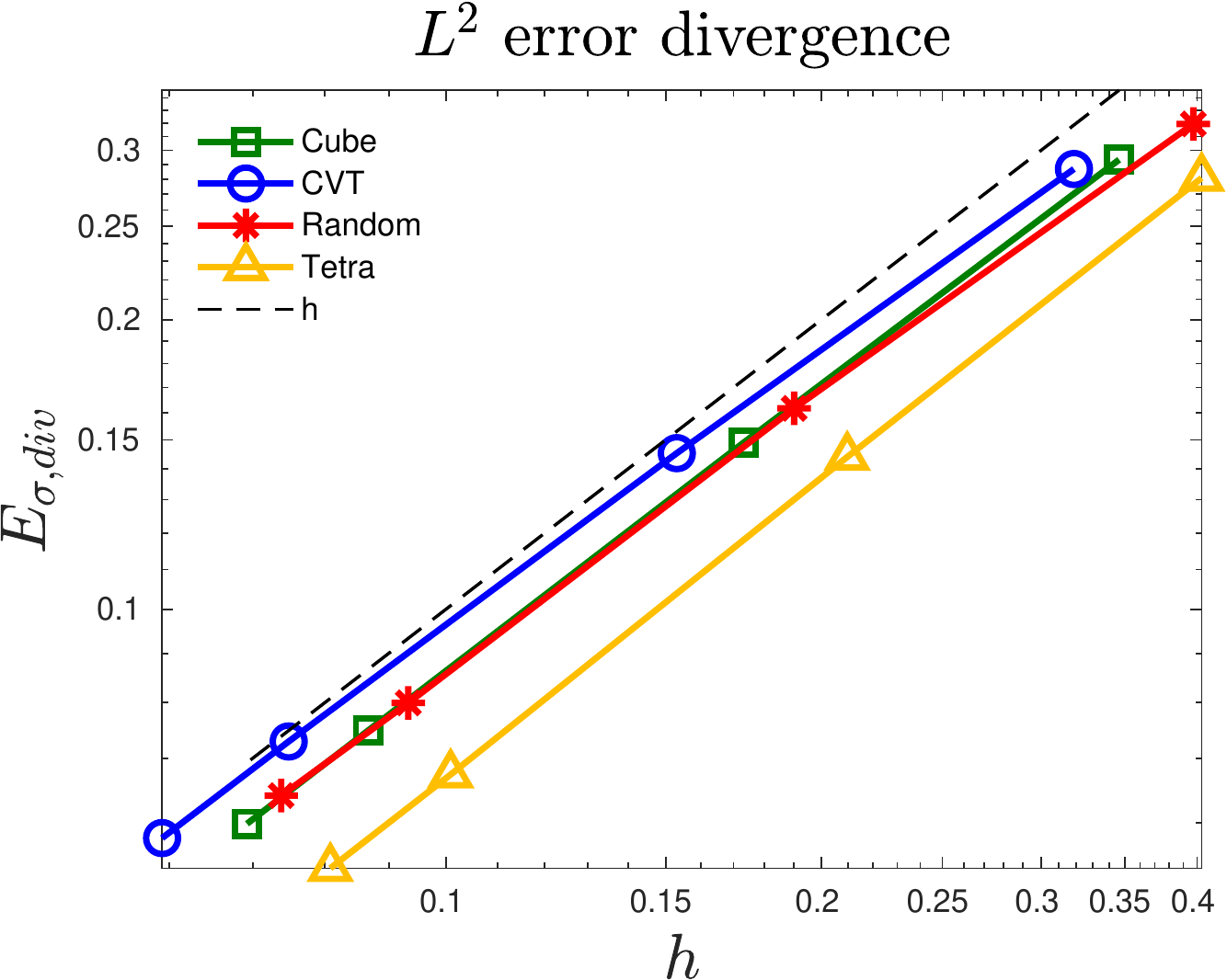}}\\
		\subfigure[]{\includegraphics[width=\sizeGraph\textwidth,trim = 0mm 0mm 0mm 0mm, clip]{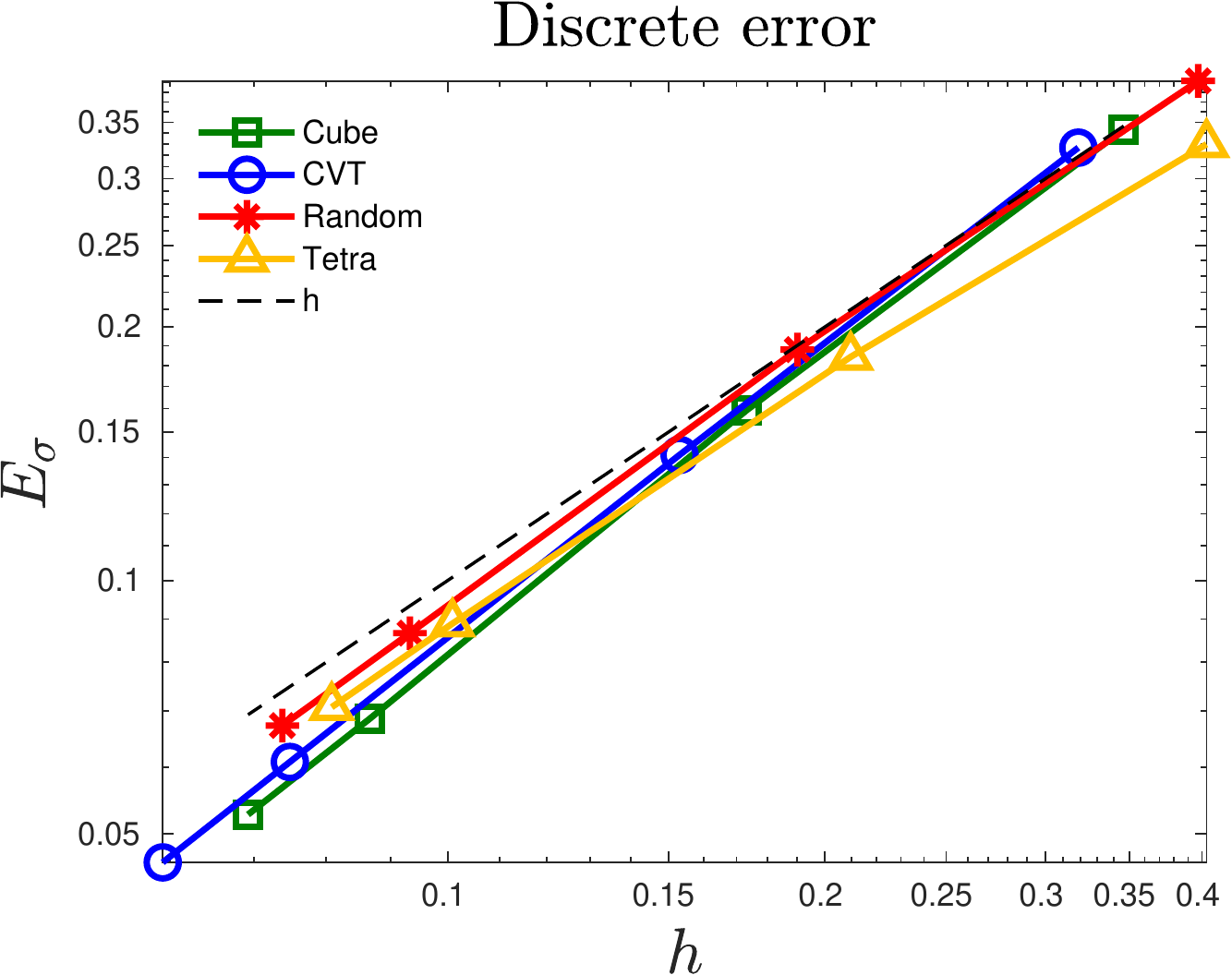}}
		\subfigure[]{\includegraphics[width=\sizeGraph\textwidth,trim = 0mm 0mm 0mm 0mm, clip]{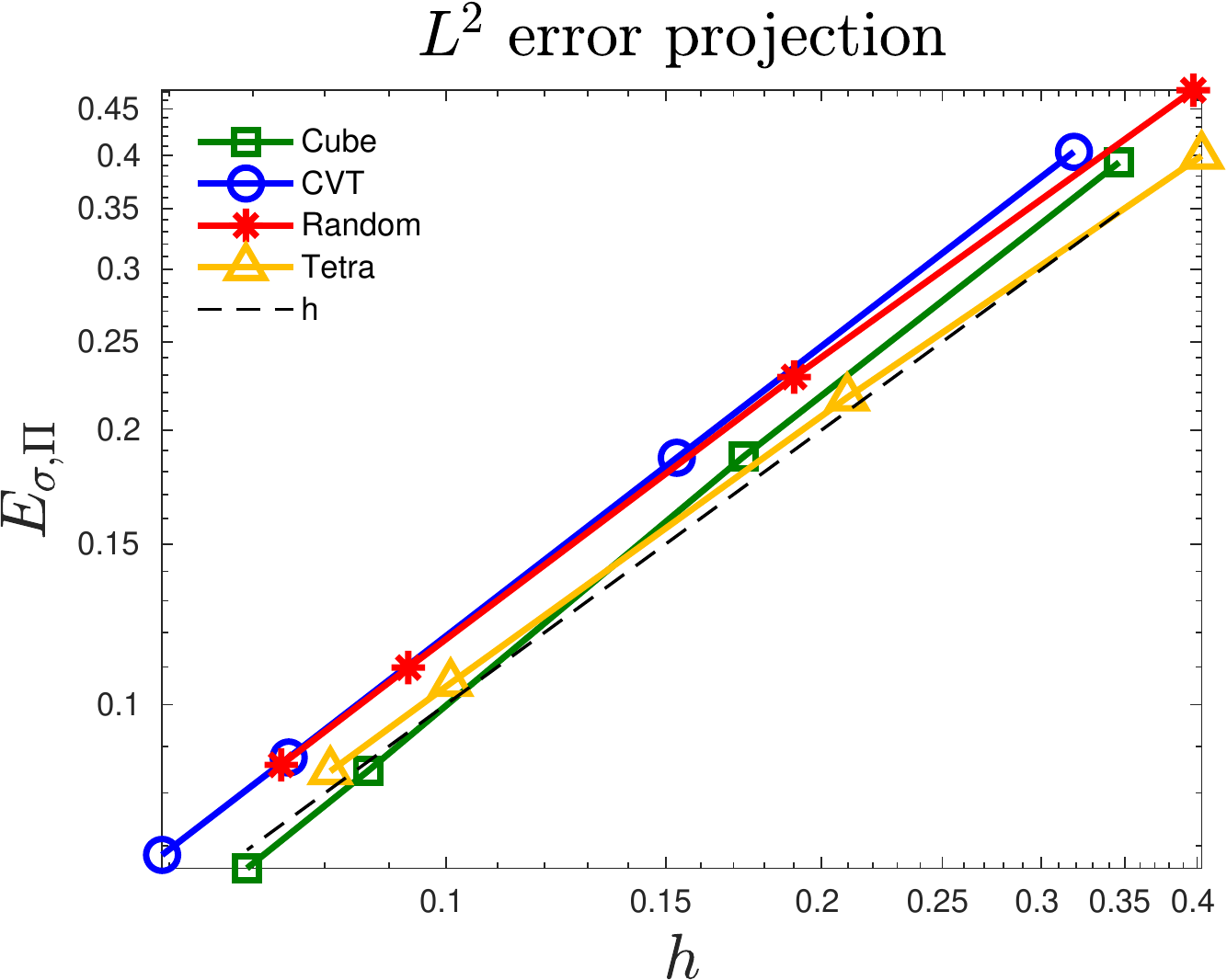}}\\
		\caption{\textbf{Example ${\bf 1}$ (compressible material):} $h$-convergence results for all meshes.}\label{fig:resuTest1}
	\end{figure}
Figure \ref{fig:resuTest1} 	reports $h$-convergence of the proposed methods for Example $1$. As expected, for the considered methods, the asymptotic convergence rate is approximately equal to $1$ for all error norms and meshes. In addition, the convergence graphs of each type of mesh are close to each others
and this fact confirms the robustness of the proposed method with respect to the element shape.
	\paragraph{Example 2 (nearly incompressible material).}\label{par:Incomprimibile}
	We consider a problem with known analytical solution. A nearly incompressible material is chosen by selecting Lam\'e 
	constants as $\lambda = 10^5$, $\mu = 0.5$. The test is designed by choosing a required solution for the displacement field and deriving the load $\bbf$ accordingly. The displacement solution is as follows:
	\begin{eqnarray*}
	\left\{
	\begin{array}{l}
	u_1={\sin(2\pi x)}^2\left(\cos(2\pi y)\sin(2\pi y){\sin(2\pi z)}^2-\cos(2\pi z)\sin(2\pi z){\sin(2\pi y)}^2 \right)\\
	u_2={\sin(2\pi y)}^2\left(\cos(2\pi z)\sin(2\pi z){\sin(2\pi x)}^2-\cos(2\pi x)\sin(2\pi x){\sin(2\pi z)}^2 \right)\\
	u_3={\sin(2\pi z)}^2\left(\cos(2\pi x)\sin(2\pi x){\sin(2\pi y)}^2-\cos(2\pi y)\sin(2\pi y){\sin(2\pi x)}^2 \right)
	
	\end{array}
	\right .
	\end{eqnarray*}
	In Figure \ref{fig:resuTest2} we report the convergence results for the proposed VEM approach. It can be clearly seen that our method shows the expected asymptotic rate of convergence for each kind of mesh.
	Moreover, also in this case the convergence lines are close to each other and this fact further confirms the robustness of the proposed scheme with respect to element shape. 
	\begin{figure}[htb!]
		\centering
		\subfigure[]{\includegraphics[width=\sizeGraph\textwidth,trim = 0mm 0mm 0mm 0mm, clip]{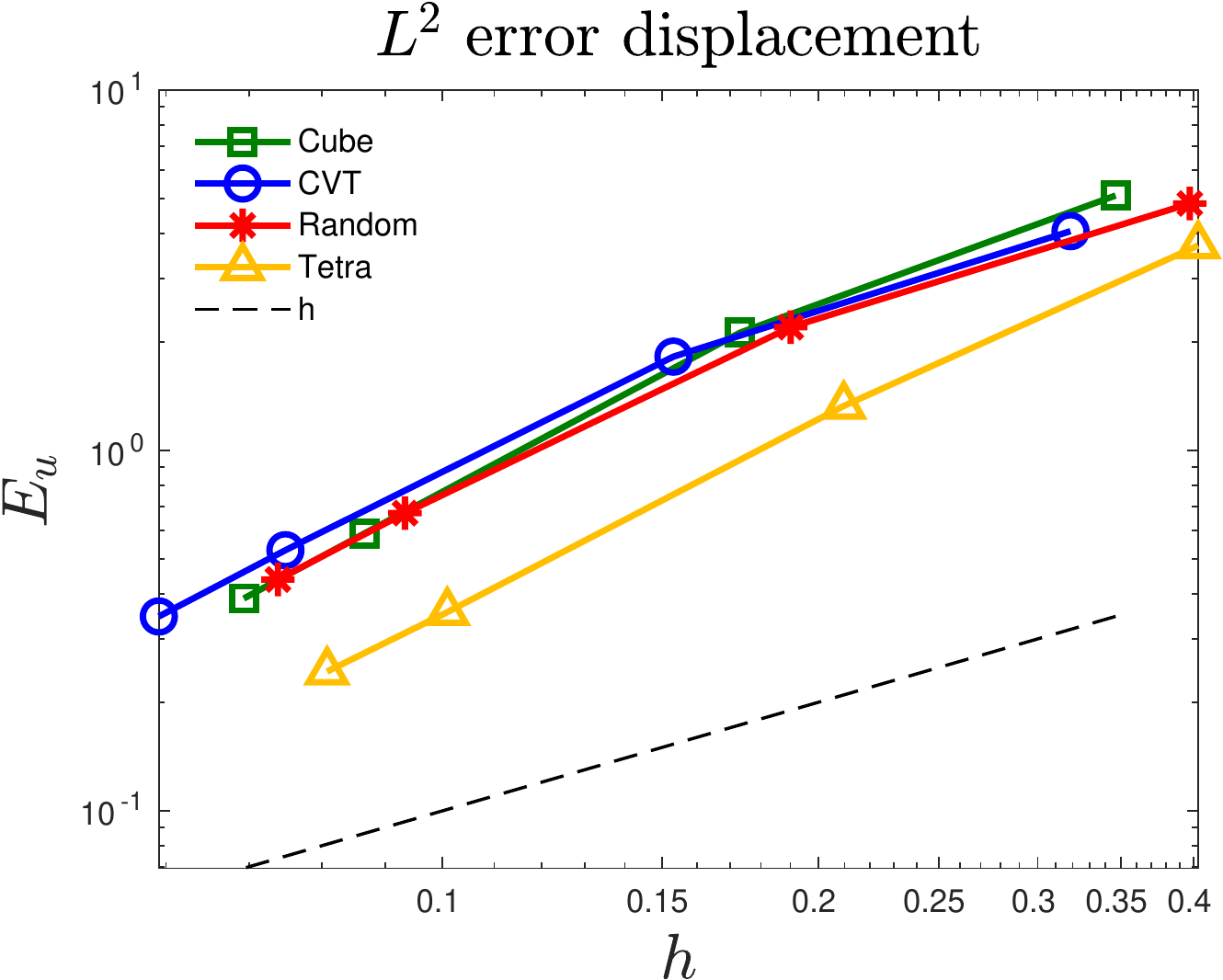}}
		\subfigure[]{\includegraphics[width=\sizeGraph\textwidth,trim = 0mm 0mm 0mm 0mm, clip]{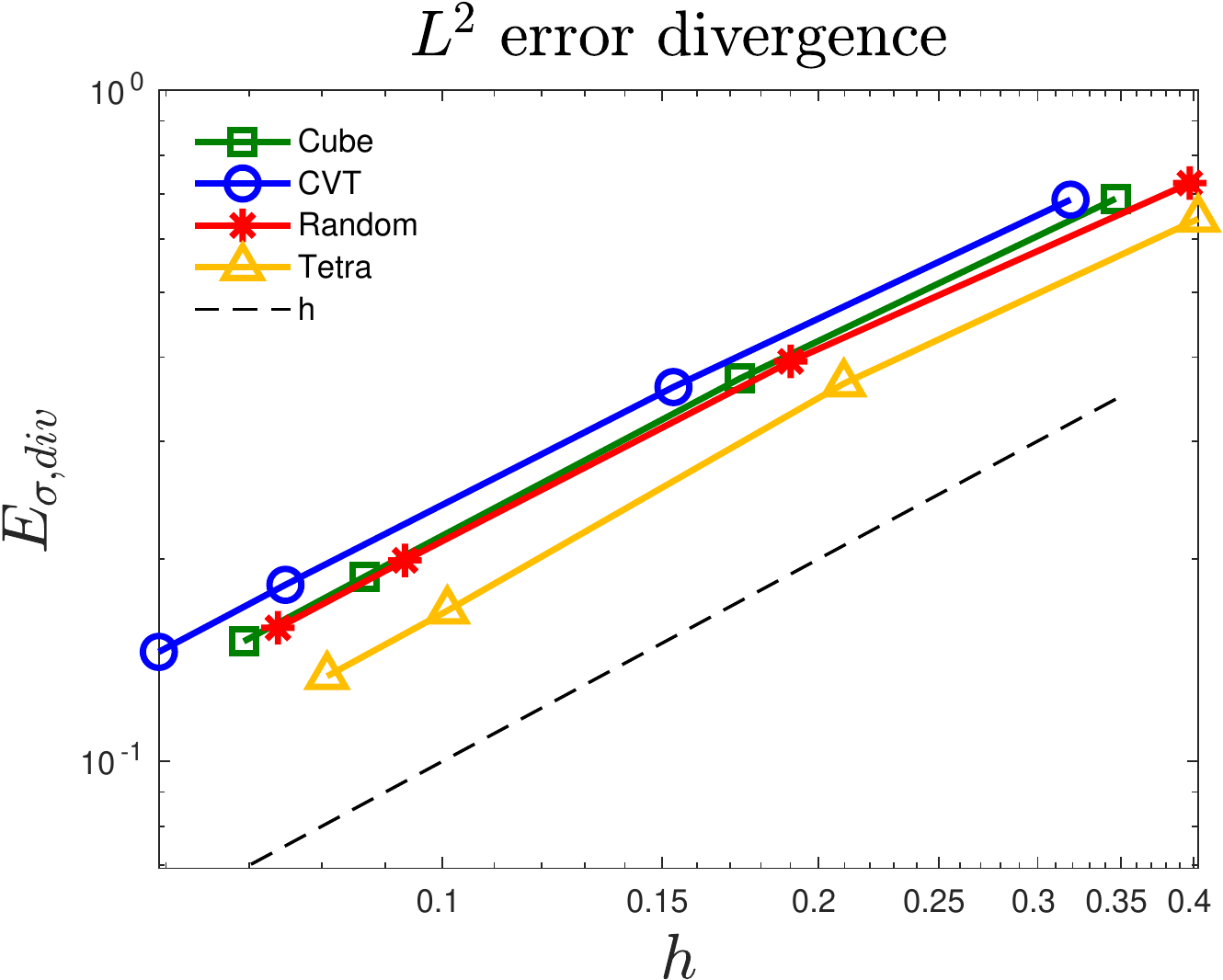}}\\
		\subfigure[]{\includegraphics[width=\sizeGraph\textwidth,trim = 0mm 0mm 0mm 0mm, clip]{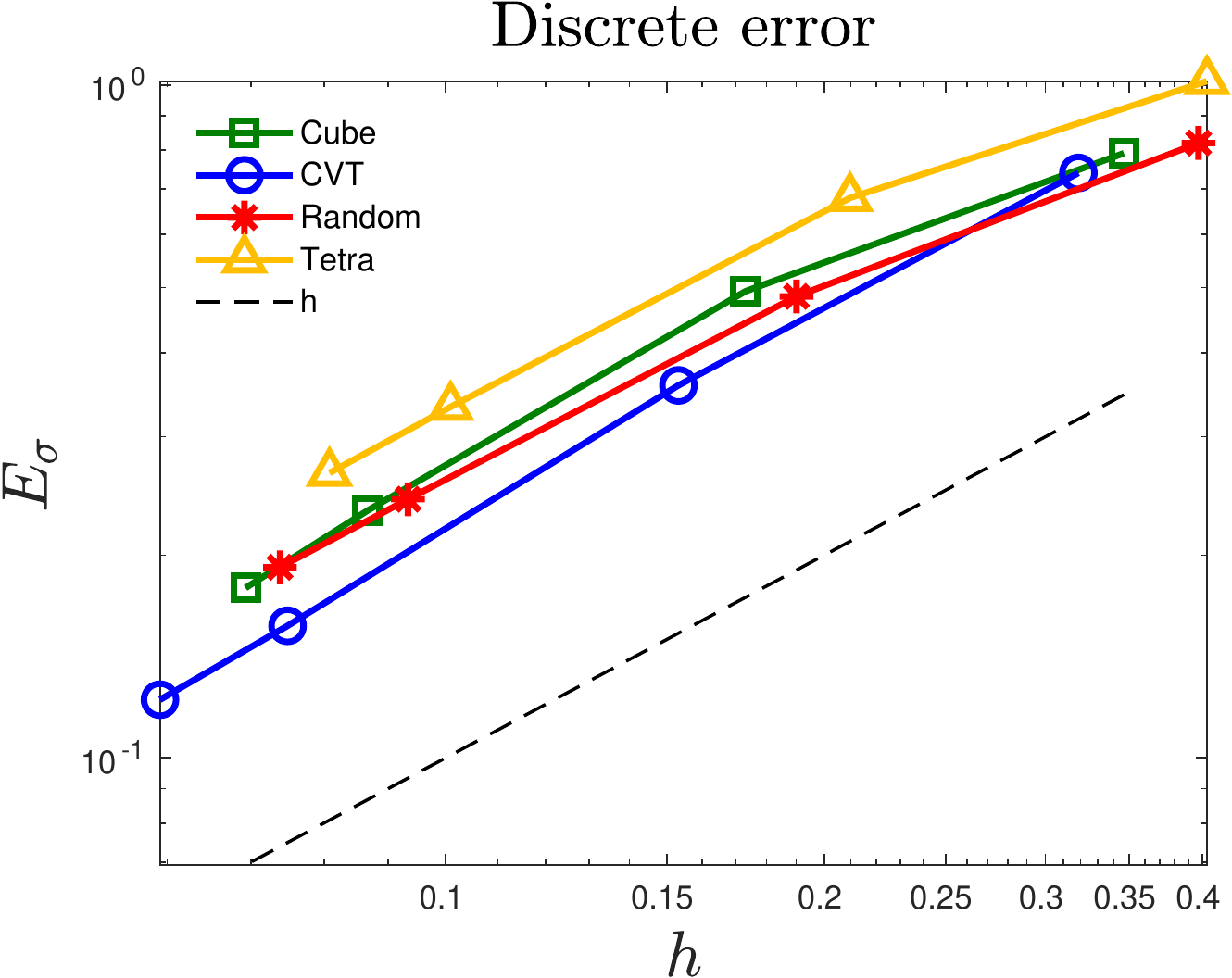}}
		\subfigure[]{\includegraphics[width=\sizeGraph\textwidth,trim = 0mm 0mm 0mm 0mm, clip]{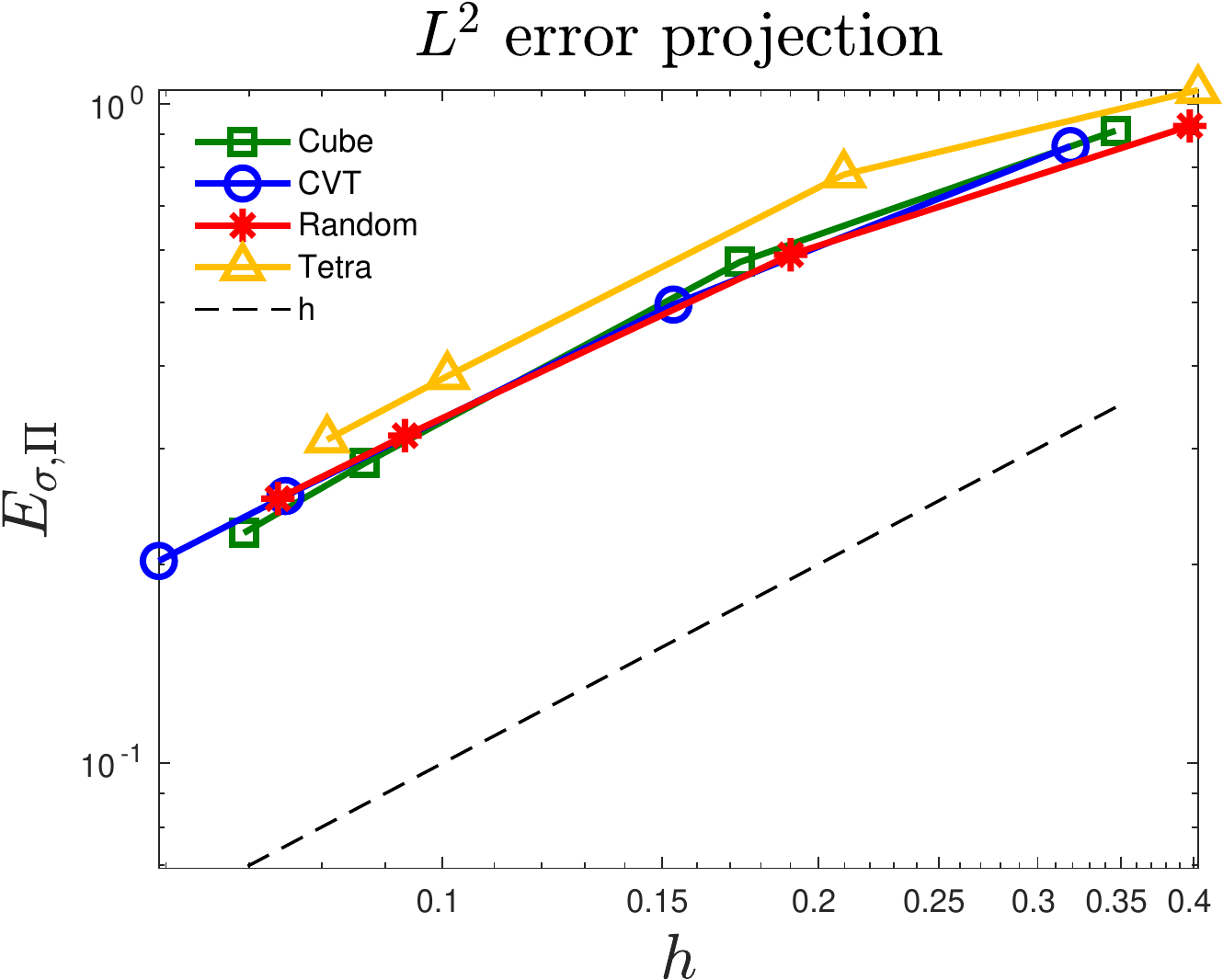}}\\
		\caption{\textbf{Example ${\bf 2}$ (nearly incompressible material):} $h$-convergence results for all meshes.\label{fig:resuTest2}}

	\end{figure}

	\paragraph{Example 3 (unloaded body).}\label{par:unloadingBody}
	We consider a problem with polynomial solution, non-homogeneous Dirichlet boundary conditions and zero loading. We take a homogeneous and isotropic material with Lam\'e constants $\lambda =1 $ and $\mu =1$ (compressible case). As in the previous examples, the test is defined by choosing a required solution and deriving the correspondig body load $\bbf$, as indicated in the following:
	\begin{eqnarray*}
	\left\{
	\begin{array}{l}
	u_1 = 2x^3-3xy^2-3xz^2\\
	u_2 = 2y^3-3yx^2-3yz^2\\
	u_3 = 2z^3-3zy^2-3zx^2\\
	\bbf=\bfzero
	\end{array}
	\right .
	\end{eqnarray*}
We remark that this is a typical example where the displacement field is nontrivial, while stresses are divergence-free. 
As expected, this latter feature is numerically satisfied by our VEM scheme.
Indeed, in Table~\ref{t:Test3} the errors $E_{\bfsigma,\bdiv}$ are close to the machine precision at each mesh refinement step. The other error behaviours are similar to the ones showed in the previous examples so we do not report such graphichs.
	\begin{table}[htb!]
		\begin{center}
			\begin{tabular}{c|c|c|c|c|}
				\cline{1-5}
				\multicolumn{1}{|c|}{Step} & {Cube}& {Tetra} & {CVT} & {Random}  \\ \cline{1-5}
				\multicolumn{1}{ |c| }{1} & 2.1904e-14 & 4.7821e-14 & 2.2369e-14 & 2.9712e-14 \\ 
				\multicolumn{1}{ |c| }{2} & 4.7351e-14 & 1.2589e-13 & 4.5054e-14 & 6.4063e-14 \\ 
				\multicolumn{1}{ |c| }{3} & 1.0024e-13 & 2.0757e-13 & 8.7751e-14 & 1.3174e-13 \\ 
				\multicolumn{1}{ |c| }{4} & 1.1793e-13 & 2.7652e-13 & 1.0922e-13 & 1.6482e-13  \\ \cline{1-5}
			\end{tabular}
			\caption{\label{t:Test3}\textbf{Example 3 (unloaded body):} $h$-convergence results for error $E_{\bfsigma, \bdiv}$.}
		\end{center}
	\end{table}


	%
	%
	%
	
	\section{Conclusions}\label{s:conclusions}
	We have proposed a Virtual Element Method for the linear elasticity 3D problem, based on the mixed Hellinger-Reissner variational principle.
	The scheme takes advantage of low-order approximation spaces for both the stresses and the displacements. In addition, the stresses are a-priori symmetric and with continuous normal component across the element interfaces.
	The convergence and stability analysis has been confirmed by some numerical results.
	A possible future development of the present paper may concern the design of schemes with reduced (minimal) degrees of freedom, by exploiting different variational principles (e.g. suitable augmented lagrangian formulations).
	
	
	
	\bibliographystyle{amsplain}
	
	\bibliography{general-bibliography,biblio,VEM}

\end{document}